\documentclass{amsart}
\usepackage[latin1]{inputenc}
\usepackage[english]{babel}
\usepackage{csquotes}

\usepackage{color} 
\usepackage[colorlinks=true,linktoc=all,linkcolor=blue,citecolor=red,filecolor=pink,urlcolor=blue]{hyperref}
\usepackage[hyperref=true,style=alphabetic,citestyle=alphabetic,sorting=nyt,backend=bibtex,maxnames=100,doi=false,isbn=false,url=false,giveninits=true]{biblatex}
\bibliography{biblio}

\usepackage{graphicx}
\graphicspath{{pictures/}}

\usepackage[mathcal,mathscr]{euscript}
\usepackage{mathrsfs}

\usepackage{tikz}
\usetikzlibrary{arrows}

\usepackage{todonotes}

\usepackage{fouriernc}
\usepackage[scaled=0.875]{helvet}
\usepackage{courier}

\usepackage{amssymb,amsmath,latexsym}
\theoremstyle{plain}                    
\newtheorem{theorem}{Theorem}[section]
\newtheorem{lemma}[theorem]{Lemma}
\newtheorem{proposition}[theorem]{Proposition}

\newtheorem{question}[theorem]{Question}
\theoremstyle{definition}

\theoremstyle{remark}
\newtheorem{remark}[theorem]{Remark}

\numberwithin{equation}{section}

\newcommand{\nn}{\mathbb N}
\newcommand{\qq}{\mathbb Q}
\newcommand{\rr}{\mathbb R}
\newcommand{\cc}{\mathbb C}
\newcommand{\hh}{\mathbb H}

\newcommand{\lk}[2]{\operatorname{lk}(#1,#2)} 
\newcommand{\st}[2]{\operatorname{st}(#1,#2)} 
\newcommand{\pcd}{\operatorname{pcd}}

\newcommand{\grbd}{\partial_\infty} 

\newcommand{\six}{\mathsf{C}_{600}}

\newcommand{\ttop}{T_{1}}
\newcommand{\tbtm}{T_{0}}
\newcommand{\interface}{I}

\newcommand{\T}{T_{21}}

\newcommand{\Sb}{\mathbb{S}}

\usepackage[inline]{enumitem}
\setenumerate[1]{label=\textbf{\arabic*.}}
\renewcommand{\leq}{\leqslant}
\renewcommand{\geq}{\geqslant}
\renewcommand{\setminus}{\smallsetminus}

  \usepackage[breakable]{tcolorbox}
    \usepackage{parskip} 

    \usepackage{caption}

    \usepackage{float}
    \usepackage{xcolor} 
    \usepackage{geometry} 
    \usepackage{textcomp} 
    \AtBeginDocument{%
    }
    \usepackage{upquote} 
    \usepackage{eurosym} 

    \usepackage{iftex}
    \ifPDFTeX
        \usepackage[T1]{fontenc}
        \IfFileExists{alphabeta.sty}{
              \usepackage{alphabeta}
          }{
              \usepackage[mathletters]{ucs}
              \usepackage[utf8x]{inputenc}
          }
    \else
        \usepackage{fontspec}
        \usepackage{unicode-math}
    \fi

    \usepackage{fancyvrb} 
    \usepackage[Export]{adjustbox} 
    \adjustboxset{max size={0.9\linewidth}{0.9\paperheight}}

    \usepackage{longtable} 
    \usepackage{booktabs}  
    \usepackage{array}     
    \usepackage{calc}      
    \usepackage[normalem]{ulem} 
 \usepackage{mathrsfs}

    \definecolor{urlcolor}{rgb}{0,.145,.698}
    \definecolor{linkcolor}{rgb}{.71,0.21,0.01}
    \definecolor{citecolor}{rgb}{.12,.54,.11}

    \definecolor{ansi-black}{HTML}{3E424D}
    \definecolor{ansi-black-intense}{HTML}{282C36}
    \definecolor{ansi-red}{HTML}{E75C58}
    \definecolor{ansi-red-intense}{HTML}{B22B31}
    \definecolor{ansi-green}{HTML}{00A250}
    \definecolor{ansi-green-intense}{HTML}{007427}
    \definecolor{ansi-yellow}{HTML}{DDB62B}
    \definecolor{ansi-yellow-intense}{HTML}{B27D12}
    \definecolor{ansi-blue}{HTML}{208FFB}
    \definecolor{ansi-blue-intense}{HTML}{0065CA}
    \definecolor{ansi-magenta}{HTML}{D160C4}
    \definecolor{ansi-magenta-intense}{HTML}{A03196}
    \definecolor{ansi-cyan}{HTML}{60C6C8}
    \definecolor{ansi-cyan-intense}{HTML}{258F8F}
    \definecolor{ansi-white}{HTML}{C5C1B4}
    \definecolor{ansi-white-intense}{HTML}{A1A6B2}
    \definecolor{ansi-default-inverse-fg}{HTML}{FFFFFF}
    \definecolor{ansi-default-inverse-bg}{HTML}{000000}

    \definecolor{outerrorbackground}{HTML}{FFDFDF}

    \DefineVerbatimEnvironment{Highlighting}{Verbatim}{commandchars=\\\{\}}

    \let\Oldtex\TeX
    \let\Oldlatex\LaTeX
    \renewcommand{\TeX}{\textrm{\Oldtex}}
    \renewcommand{\LaTeX}{\textrm{\Oldlatex}}

\makeatletter
\def\PY@reset{\let\PY@it=\relax \let\PY@bf=\relax%
    \let\PY@ul=\relax \let\PY@tc=\relax%
    \let\PY@bc=\relax \let\PY@ff=\relax}
\def\PY@tok#1{\csname PY@tok@#1\endcsname}
\def\PY@toks#1+{\ifx\relax#1\empty\else%
    \PY@tok{#1}\expandafter\PY@toks\fi}
\def\PY@do#1{\PY@bc{\PY@tc{\PY@ul{%
    \PY@it{\PY@bf{\PY@ff{#1}}}}}}}
\def\PY#1#2{\PY@reset\PY@toks#1+\relax+\PY@do{#2}}

\@namedef{PY@tok@w}{\def\PY@tc##1{\textcolor[rgb]{0.73,0.73,0.73}{##1}}}
\@namedef{PY@tok@c}{\let\PY@it=\textit\def\PY@tc##1{\textcolor[rgb]{0.24,0.48,0.48}{##1}}}
\@namedef{PY@tok@cp}{\def\PY@tc##1{\textcolor[rgb]{0.61,0.40,0.00}{##1}}}
\@namedef{PY@tok@k}{\let\PY@bf=\textbf\def\PY@tc##1{\textcolor[rgb]{0.00,0.50,0.00}{##1}}}
\@namedef{PY@tok@kp}{\def\PY@tc##1{\textcolor[rgb]{0.00,0.50,0.00}{##1}}}
\@namedef{PY@tok@kt}{\def\PY@tc##1{\textcolor[rgb]{0.69,0.00,0.25}{##1}}}
\@namedef{PY@tok@o}{\def\PY@tc##1{\textcolor[rgb]{0.40,0.40,0.40}{##1}}}
\@namedef{PY@tok@ow}{\let\PY@bf=\textbf\def\PY@tc##1{\textcolor[rgb]{0.67,0.13,1.00}{##1}}}
\@namedef{PY@tok@nb}{\def\PY@tc##1{\textcolor[rgb]{0.00,0.50,0.00}{##1}}}
\@namedef{PY@tok@nf}{\def\PY@tc##1{\textcolor[rgb]{0.00,0.00,1.00}{##1}}}
\@namedef{PY@tok@nc}{\let\PY@bf=\textbf\def\PY@tc##1{\textcolor[rgb]{0.00,0.00,1.00}{##1}}}
\@namedef{PY@tok@nn}{\let\PY@bf=\textbf\def\PY@tc##1{\textcolor[rgb]{0.00,0.00,1.00}{##1}}}
\@namedef{PY@tok@ne}{\let\PY@bf=\textbf\def\PY@tc##1{\textcolor[rgb]{0.80,0.25,0.22}{##1}}}
\@namedef{PY@tok@nv}{\def\PY@tc##1{\textcolor[rgb]{0.10,0.09,0.49}{##1}}}
\@namedef{PY@tok@no}{\def\PY@tc##1{\textcolor[rgb]{0.53,0.00,0.00}{##1}}}
\@namedef{PY@tok@nl}{\def\PY@tc##1{\textcolor[rgb]{0.46,0.46,0.00}{##1}}}
\@namedef{PY@tok@ni}{\let\PY@bf=\textbf\def\PY@tc##1{\textcolor[rgb]{0.44,0.44,0.44}{##1}}}
\@namedef{PY@tok@na}{\def\PY@tc##1{\textcolor[rgb]{0.41,0.47,0.13}{##1}}}
\@namedef{PY@tok@nt}{\let\PY@bf=\textbf\def\PY@tc##1{\textcolor[rgb]{0.00,0.50,0.00}{##1}}}
\@namedef{PY@tok@nd}{\def\PY@tc##1{\textcolor[rgb]{0.67,0.13,1.00}{##1}}}
\@namedef{PY@tok@s}{\def\PY@tc##1{\textcolor[rgb]{0.73,0.13,0.13}{##1}}}
\@namedef{PY@tok@sd}{\let\PY@it=\textit\def\PY@tc##1{\textcolor[rgb]{0.73,0.13,0.13}{##1}}}
\@namedef{PY@tok@si}{\let\PY@bf=\textbf\def\PY@tc##1{\textcolor[rgb]{0.64,0.35,0.47}{##1}}}
\@namedef{PY@tok@se}{\let\PY@bf=\textbf\def\PY@tc##1{\textcolor[rgb]{0.67,0.36,0.12}{##1}}}
\@namedef{PY@tok@sr}{\def\PY@tc##1{\textcolor[rgb]{0.64,0.35,0.47}{##1}}}
\@namedef{PY@tok@ss}{\def\PY@tc##1{\textcolor[rgb]{0.10,0.09,0.49}{##1}}}
\@namedef{PY@tok@sx}{\def\PY@tc##1{\textcolor[rgb]{0.00,0.50,0.00}{##1}}}
\@namedef{PY@tok@m}{\def\PY@tc##1{\textcolor[rgb]{0.40,0.40,0.40}{##1}}}
\@namedef{PY@tok@gh}{\let\PY@bf=\textbf\def\PY@tc##1{\textcolor[rgb]{0.00,0.00,0.50}{##1}}}
\@namedef{PY@tok@gu}{\let\PY@bf=\textbf\def\PY@tc##1{\textcolor[rgb]{0.50,0.00,0.50}{##1}}}
\@namedef{PY@tok@gd}{\def\PY@tc##1{\textcolor[rgb]{0.63,0.00,0.00}{##1}}}
\@namedef{PY@tok@gi}{\def\PY@tc##1{\textcolor[rgb]{0.00,0.52,0.00}{##1}}}
\@namedef{PY@tok@gr}{\def\PY@tc##1{\textcolor[rgb]{0.89,0.00,0.00}{##1}}}
\@namedef{PY@tok@ge}{\let\PY@it=\textit}
\@namedef{PY@tok@gs}{\let\PY@bf=\textbf}
\@namedef{PY@tok@ges}{\let\PY@bf=\textbf\let\PY@it=\textit}
\@namedef{PY@tok@gp}{\let\PY@bf=\textbf\def\PY@tc##1{\textcolor[rgb]{0.00,0.00,0.50}{##1}}}
\@namedef{PY@tok@go}{\def\PY@tc##1{\textcolor[rgb]{0.44,0.44,0.44}{##1}}}
\@namedef{PY@tok@gt}{\def\PY@tc##1{\textcolor[rgb]{0.00,0.27,0.87}{##1}}}
\@namedef{PY@tok@err}{\def\PY@bc##1{{\setlength{\fboxsep}{\string -\fboxrule}\fcolorbox[rgb]{1.00,0.00,0.00}{1,1,1}{\strut ##1}}}}
\@namedef{PY@tok@kc}{\let\PY@bf=\textbf\def\PY@tc##1{\textcolor[rgb]{0.00,0.50,0.00}{##1}}}
\@namedef{PY@tok@kd}{\let\PY@bf=\textbf\def\PY@tc##1{\textcolor[rgb]{0.00,0.50,0.00}{##1}}}
\@namedef{PY@tok@kn}{\let\PY@bf=\textbf\def\PY@tc##1{\textcolor[rgb]{0.00,0.50,0.00}{##1}}}
\@namedef{PY@tok@kr}{\let\PY@bf=\textbf\def\PY@tc##1{\textcolor[rgb]{0.00,0.50,0.00}{##1}}}
\@namedef{PY@tok@bp}{\def\PY@tc##1{\textcolor[rgb]{0.00,0.50,0.00}{##1}}}
\@namedef{PY@tok@fm}{\def\PY@tc##1{\textcolor[rgb]{0.00,0.00,1.00}{##1}}}
\@namedef{PY@tok@vc}{\def\PY@tc##1{\textcolor[rgb]{0.10,0.09,0.49}{##1}}}
\@namedef{PY@tok@vg}{\def\PY@tc##1{\textcolor[rgb]{0.10,0.09,0.49}{##1}}}
\@namedef{PY@tok@vi}{\def\PY@tc##1{\textcolor[rgb]{0.10,0.09,0.49}{##1}}}
\@namedef{PY@tok@vm}{\def\PY@tc##1{\textcolor[rgb]{0.10,0.09,0.49}{##1}}}
\@namedef{PY@tok@sa}{\def\PY@tc##1{\textcolor[rgb]{0.73,0.13,0.13}{##1}}}
\@namedef{PY@tok@sb}{\def\PY@tc##1{\textcolor[rgb]{0.73,0.13,0.13}{##1}}}
\@namedef{PY@tok@sc}{\def\PY@tc##1{\textcolor[rgb]{0.73,0.13,0.13}{##1}}}
\@namedef{PY@tok@dl}{\def\PY@tc##1{\textcolor[rgb]{0.73,0.13,0.13}{##1}}}
\@namedef{PY@tok@s2}{\def\PY@tc##1{\textcolor[rgb]{0.73,0.13,0.13}{##1}}}
\@namedef{PY@tok@sh}{\def\PY@tc##1{\textcolor[rgb]{0.73,0.13,0.13}{##1}}}
\@namedef{PY@tok@s1}{\def\PY@tc##1{\textcolor[rgb]{0.73,0.13,0.13}{##1}}}
\@namedef{PY@tok@mb}{\def\PY@tc##1{\textcolor[rgb]{0.40,0.40,0.40}{##1}}}
\@namedef{PY@tok@mf}{\def\PY@tc##1{\textcolor[rgb]{0.40,0.40,0.40}{##1}}}
\@namedef{PY@tok@mh}{\def\PY@tc##1{\textcolor[rgb]{0.40,0.40,0.40}{##1}}}
\@namedef{PY@tok@mi}{\def\PY@tc##1{\textcolor[rgb]{0.40,0.40,0.40}{##1}}}
\@namedef{PY@tok@il}{\def\PY@tc##1{\textcolor[rgb]{0.40,0.40,0.40}{##1}}}
\@namedef{PY@tok@mo}{\def\PY@tc##1{\textcolor[rgb]{0.40,0.40,0.40}{##1}}}
\@namedef{PY@tok@ch}{\let\PY@it=\textit\def\PY@tc##1{\textcolor[rgb]{0.24,0.48,0.48}{##1}}}
\@namedef{PY@tok@cm}{\let\PY@it=\textit\def\PY@tc##1{\textcolor[rgb]{0.24,0.48,0.48}{##1}}}
\@namedef{PY@tok@cpf}{\let\PY@it=\textit\def\PY@tc##1{\textcolor[rgb]{0.24,0.48,0.48}{##1}}}
\@namedef{PY@tok@c1}{\let\PY@it=\textit\def\PY@tc##1{\textcolor[rgb]{0.24,0.48,0.48}{##1}}}
\@namedef{PY@tok@cs}{\let\PY@it=\textit\def\PY@tc##1{\textcolor[rgb]{0.24,0.48,0.48}{##1}}}

\makeatother

    \makeatletter
        \newbox\Wrappedcontinuationbox 
        \newbox\Wrappedvisiblespacebox 
        \newcommand*\Wrappedvisiblespace {\textcolor{red}{\textvisiblespace}} 
        \newcommand*\Wrappedcontinuationsymbol {\textcolor{red}{\llap{\tiny$\m@th\hookrightarrow$}}} 
        \newcommand*\Wrappedcontinuationindent {3ex } 
        \newcommand*\Wrappedafterbreak {\kern\Wrappedcontinuationindent\copy\Wrappedcontinuationbox} 
        \newcommand*\Wrappedbreaksatspecials {%
            \def\PYGZus{\discretionary{\char`\_}{\Wrappedafterbreak}{\char`\_}}%
            \def\PYGZob{\discretionary{}{\Wrappedafterbreak\char`\{}{\char`\{}}%
            \def\PYGZcb{\discretionary{\char`\}}{\Wrappedafterbreak}{\char`\}}}%
            \def\PYGZca{\discretionary{\char`\^}{\Wrappedafterbreak}{\char`\^}}%
            \def\PYGZam{\discretionary{\char`\&}{\Wrappedafterbreak}{\char`\&}}%
            \def\PYGZlt{\discretionary{}{\Wrappedafterbreak\char`\<}{\char`\<}}%
            \def\PYGZgt{\discretionary{\char`\>}{\Wrappedafterbreak}{\char`\>}}%
            \def\PYGZsh{\discretionary{}{\Wrappedafterbreak\char`\#}{\char`\#}}%
            \def\PYGZpc{\discretionary{}{\Wrappedafterbreak\char`\%}{\char`\%}}%
            \def\PYGZdl{\discretionary{}{\Wrappedafterbreak\char`\$}{\char`\$}}%
            \def\PYGZhy{\discretionary{\char`\-}{\Wrappedafterbreak}{\char`\-}}%
            \def\PYGZsq{\discretionary{}{\Wrappedafterbreak\textquotesingle}{\textquotesingle}}%
            \def\PYGZdq{\discretionary{}{\Wrappedafterbreak\char`\"}{\char`\"}}%
            \def\PYGZti{\discretionary{\char`\~}{\Wrappedafterbreak}{\char`\~}}%
        } 
        \newcommand*\Wrappedbreaksatpunct {%
            \lccode`\~`\.\lowercase{\def~}{\discretionary{\hbox{\char`\.}}{\Wrappedafterbreak}{\hbox{\char`\.}}}%
            \lccode`\~`\,\lowercase{\def~}{\discretionary{\hbox{\char`\,}}{\Wrappedafterbreak}{\hbox{\char`\,}}}%
            \lccode`\~`\;\lowercase{\def~}{\discretionary{\hbox{\char`\;}}{\Wrappedafterbreak}{\hbox{\char`\;}}}%
            \lccode`\~`\:\lowercase{\def~}{\discretionary{\hbox{\char`\:}}{\Wrappedafterbreak}{\hbox{\char`\:}}}%
            \lccode`\~`\?\lowercase{\def~}{\discretionary{\hbox{\char`\?}}{\Wrappedafterbreak}{\hbox{\char`\?}}}%
            \lccode`\~`\!\lowercase{\def~}{\discretionary{\hbox{\char`\!}}{\Wrappedafterbreak}{\hbox{\char`\!}}}%
            \lccode`\~`\/\lowercase{\def~}{\discretionary{\hbox{\char`\/}}{\Wrappedafterbreak}{\hbox{\char`\/}}}%
            \catcode`\.\active
            \catcode`\,\active 
            \catcode`\;\active
            \catcode`\:\active
            \catcode`\?\active
            \catcode`\!\active
            \catcode`\/\active 
            \lccode`\~`\~ 	
        }
    \makeatother

    \let\OriginalVerbatim=\Verbatim
    \makeatletter
    \renewcommand{\Verbatim}[1][1]{%
        \sbox\Wrappedcontinuationbox {\Wrappedcontinuationsymbol}%
        \sbox\Wrappedvisiblespacebox {\FV@SetupFont\Wrappedvisiblespace}%
        \def\FancyVerbFormatLine ##1{\hsize\linewidth
            \vtop{\raggedright\hyphenpenalty\z@\exhyphenpenalty\z@
                \doublehyphendemerits\z@\finalhyphendemerits\z@
                \strut ##1\strut}%
        }%
        \def\FV@Space {%
            \nobreak\hskip\z@ plus\fontdimen3\font minus\fontdimen4\font
            \discretionary{\copy\Wrappedvisiblespacebox}{\Wrappedafterbreak}
            {\kern\fontdimen2\font}%
        }%
        
        \Wrappedbreaksatspecials
        \OriginalVerbatim[#1,codes*=\Wrappedbreaksatpunct]%
    }
    \makeatother

    \definecolor{incolor}{HTML}{303F9F}
    \definecolor{outcolor}{HTML}{D84315}
    \definecolor{cellborder}{HTML}{CFCFCF}
    \definecolor{cellbackground}{HTML}{F7F7F7}
    
    \makeatletter
    \newcommand{\boxspacing}{\kern\kvtcb@left@rule\kern\kvtcb@boxsep}
    \makeatother

\makeatletter
\@namedef{subjclassname@2020}{\textup{2020} Mathematics Subject Classification}
\makeatother

\title[Convex cocompact subgroups with limit set a Pontryagin sphere]{Convex cocompact groups in real hyperbolic spaces with limit set a Pontryagin sphere}
\author{Sami Douba}
\address{Mathematisches Institut der Universit\"at Bonn, Endenicher Allee 60, 53115 Bonn, Germany}
\email{douba@math.uni-bonn.de}

\author{Gye-Seon Lee}
\address{Department of Mathematical Sciences and Research Institute of Mathematics, Seoul National University, Seoul 08826, South Korea}
\email{gyeseonlee@snu.ac.kr}

\author{Ludovic Marquis}
\address{Univ Rennes, CNRS, IRMAR - UMR 6625, F-35000 Rennes, France}
\email{ludovic.marquis@univ-rennes.fr}

\author{Lorenzo Ruffoni}
\address{Department of Mathematics and Statistics - Binghamton University, Binghamton, NY 13902, USA}
\email{lorenzo.ruffoni2@gmail.com}

\begin{document}

\date{\today}
\subjclass[2020]{20F65, 20F67, 22E40 (primary); 20F55 (secondary)}

\begin{abstract}
We exhibit two examples of convex cocompact subgroups of the isometry groups of real hyperbolic spaces with limit set a Pontryagin sphere: one generated by $50$ reflections of $\mathbb{H}^4$, and the other by a rotation of order $21$ and a reflection of $\mathbb{H}^6$. For each of them, we also locate convex cocompact subgroups  with limit set a Menger curve.
\end{abstract}

\keywords{reflection groups, convex cocompact subgroups, real hyperbolic spaces, limit sets, Pontryagin sphere, Menger curve, right-angled Coxeter groups}

\maketitle

\section{Introduction}

For any Gromov hyperbolic group $G$, the Gromov boundary $\partial_\infty G$ of $G$ is a compact metrizable space \cite{G87}. A natural problem is to characterize which spaces can arise as boundaries of Gromov hyperbolic groups (see \cite[Question A]{KK00}). 
When $\partial_\infty G$ is $1$-dimensional, if $G$ does not split over virtually cyclic subgroups, then $\partial_\infty G$ can be either $S^1$, the Sierpi\'nski carpet, or the Menger curve (see \cite[Theorem 4]{KK00}). Moreover, the Menger curve is the generic case, see \cite{CH95}.
In dimensions at least $2$, no such characterization is available. 
Higher-dimensional spheres and Sierpi\'nski compacta are realized by the fundamental group of real hyperbolic manifolds, respectively closed or with totally geodesic boundary. Other spaces known to be realized include certain Menger compacta, the Pontryagin surfaces, and many trees of manifolds, including the Pontryagin sphere (see \cite{DR99,FI03,DO07,PS09,SW20}).

Another interesting question concerns which Gromov hyperbolic groups can be realized as discrete subgroups $\Gamma < \mathrm{Isom} (\hh^n)$ for some $n$, or which spaces can occur as the limit set of such groups $\Gamma$. The limit set $\Lambda_{\Gamma}$ of $\Gamma$ is a fundamental object in the study of Kleinian groups, and there is a rich theory relating the Hausdorff dimension of $\Lambda_{\Gamma}$ to the critical exponent of $\Gamma$, see \cite{KAP09}. In particular, when $\Gamma$ is {\em convex cocompact}, these invariants  coincide, see \cite{SU79}. Moreover, in this case, the abstract group $\Gamma$ is Gromov hyperbolic, and $\Lambda_\Gamma$ is $\Gamma$-equivariantly homeomorphic to $\partial_\infty \Gamma$. 
However, very few topological spaces are currently known to occur as limit sets. For instance, it remains an open question whether there exists a convex cocompact subgroup $\Gamma < \mathrm{Isom} (\hh^n)$ whose limit set is a \v Cech cohomology sphere other than the standard sphere (see \cite[Question 9.5]{Kapovich08} and \cite{BM91}.)

We now turn to a more detailed discussion of the main objectives of this paper.
 The  \textit{Sierpi\'nski carpet} is the fractal curve obtained from a square by removing smaller and smaller squares, similarly to  the way one obtains the Cantor set  from an interval.
In particular, the Sierpi\'nski carpet  has a well-defined collection of boundary squares.
The \textit{Pontryagin sphere} is obtained from the Sierpi\'nski carpet  by identifying the opposite sides of each boundary square as one would glue the edges of a square to form a torus.
This results in a $2$-dimensional compact metrizable space without non-empty open subsets embeddable in  $\rr^2$; see Figure~\ref{fig:pontryagin}.
The Pontryagin sphere can also be described as the  limit of a certain inverse system of closed connected  orientable surfaces of positive genus; see \cite{JA91,FI03,ZA10,SW20,SW20sp} and \S\ref{sec:pontryagin}. 

The Pontryagin sphere is known to appear as the Gromov boundary of many Gromov hyperbolic groups, see for instance \cite[Remark 4.4]{PS09} and \cite{JS06,ZA10}, and also more recent work \cite{FGLS25, CDSS25} where infinitely many quasi-isometry classes of such groups are constructed.
Here we consider the problem of realizing the Pontryagin sphere as the limit set of concrete reflection groups in the isometry group $\mathrm{Isom} (\hh^d)$ of a real hyperbolic space $\hh^d$.

\begin{theorem}\label{main4}
There is a convex cocompact right-angled reflection subgroup $\Gamma^4$ of $\mathrm{Isom} (\hh^4)$   whose limit set is homeomorphic to the Pontryagin sphere.
\end{theorem}

\begin{theorem}\label{main6}
There is a 2-generated convex cocompact subgroup $\Gamma_2^6$ of $\mathrm{Isom} (\hh^6)$ whose limit set is homeomorphic to the Pontryagin sphere. Moreover, $\Gamma_2^6$ contains a normal subgroup $\Gamma^6$ of index $21$ which is a right-angled reflection group.
\end{theorem}

The starting point of our work is the following. If $W$ is a Gromov hyperbolic right-angled Coxeter group (RACG) whose nerve is a closed connected orientable surface of positive genus then the Gromov boundary $\partial_\infty W$ of $W$ is a Pontryagin sphere; see \cite{FI03}. 
Hence, the upshot of this paper is to find convex cocompact actions of two such Coxeter groups on $\hh^4$ and $\hh^6$, respectively.

The groups constructed in this paper are optimal in the following sense.
The group in dimension $d=4$ appears in the smallest possible dimension, as the Pontryagin sphere does not embed in $\Sb^2=\partial_\infty \hh^3$.
In particular, $\Gamma^4$ is Zariski dense in  $\mathrm{Isom} (\hh^4)$.
On the other hand, the group $\Gamma_2^6$ in dimension $d=6$ has the smallest possible number of generators (as the limit set of a cyclic group consists of at most $2$ points).
Moreover, the groups in dimension $d=6$ are Zariski dense, so their limit set is not contained in any round $\Sb^k\subseteq \Sb^5=\partial_\infty \hh^6$ for $k\leq 4$; see Remark~\ref{rem:zariski_dense}.

The \textit{Menger curve} is a  fractal curve that is obtained from a cube by removing smaller and smaller cubes, similarly to the way one obtains the  Sierpi\'nski carpet from a square and  the Cantor set from an interval.
It can be characterized as the unique compact, metrizable, connected, locally connected, $1$-dimensional space without local cut points and without non-empty open subsets embeddable in $\rr^2$; see \cite{AN58,AN58b}.
According to \cite{KP96}, the Pontryagin sphere has a dense collection of Menger curves, so it is natural to look for Menger curves that are limit sets of quasiconvex subgroups of $\Gamma^4$ or $\Gamma^6$.
In Theorem~\ref{thm:menger} we  show that both $\Gamma^4$ and $\Gamma^6$ have convex cocompact subgroups with limit set homeomorphic to the Menger curve. 
These subgroups are obtained by forgetting one (or more, suitably chosen) reflections. 
Previous examples of convex cocompact  subgroups of $\mathrm{Isom}(\hh^4)$ with limit set a Menger curve were constructed by Bourdon \cite{BO97}, see Remark~\ref{remark:bourdon}.

We remark that indeed the answer to the following question does not appear to be known.

\begin{question}
Does every Gromov hyperbolic group with boundary a Pontryagin sphere possess a quasiconvex subgroup with boundary a Menger curve?
\end{question}

Note that the analogous statement for Gromov hyperbolic groups with boundary a $2$-sphere, namely, that every group of the latter form possesses a quasiconvex subgroup with boundary a Sierpi\'nski carpet, would follow from the Cannon conjecture, but even this implication requires non-trivial results about hyperbolic $3$-manifolds, including several of the ingredients that went into the solution to the virtual Haken conjecture~\cite{Ber15}.

\begin{figure}[ht]
    \centering
    \includegraphics[scale=.8]{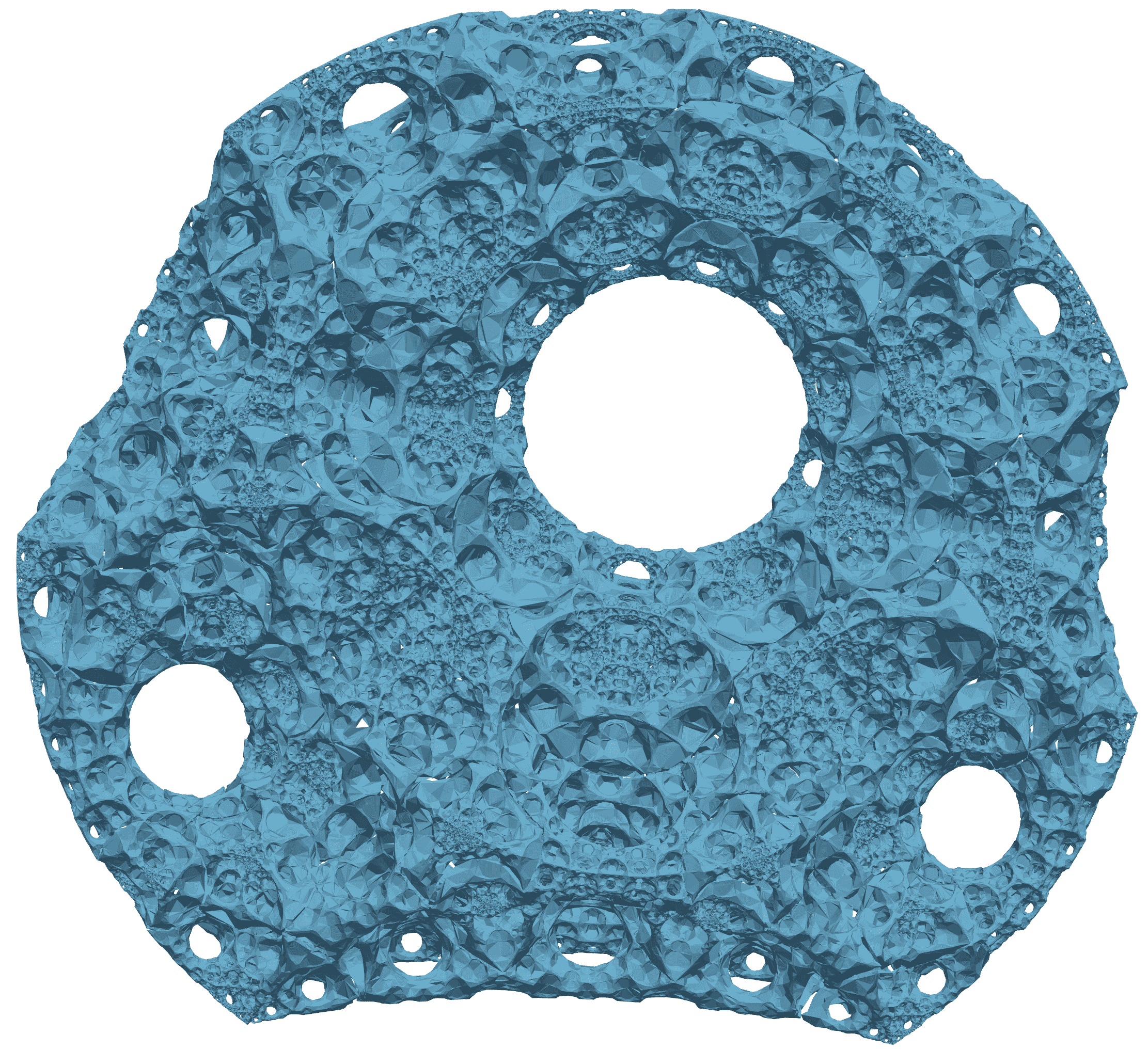}
    \caption{A rendering of the Pontryagin sphere, courtesy of Theodore Weisman. This is in fact a plot of the limit set of the group $\Gamma^4$ in Theorem~\ref{main4}.}
    \label{fig:pontryagin}
\end{figure}


We conclude with a few comments on realizing the Pontryagin sphere as a limit set in other spaces.
By embedding $\hh^4$ as a totally geodesic subspace in higher-dimensional hyperbolic spaces, one can realize $\Gamma^4$ as a convex cocompact subgroup of  $\mathrm{Isom}(\hh^d)$ for each $d\geq 4$. 
Similarly, since  $\hh^4$ embeds as a totally geodesic subspace in the complex hyperbolic spaces $\hh^d_\mathbb{C}$ for $d\geq 4$, in all quaternionic hyperbolic spaces, and in the octonionic hyperbolic plane, all these symmetric spaces have convex cocompact groups of isometries with limit set a Pontryagin sphere.
This leaves us with $\hh_\cc^2$ and $\hh_\cc^3$.
Jeffrey Danciger and Theodore Weisman
have communicated to us that they can show that the deformations of certain cusped hyperbolic $3$-orbifold groups inside $\mathrm{Isom}(\hh_\cc^3)$ constructed in \cite{PT23} descend to convex cocompact representations of suitable quotients with limit set a Pontryagin sphere.
On the other hand, it is known for instance that the limit set of a convex cocompact subgroup of $\mathrm{Isom}(\hh_\cc^2)$ cannot have topological dimension $2$, see \cite[Theorem 17]{KAP22}. In particular, the limit set of such a group cannot be a Pontryagin sphere.
To summarize, there is a complete list of the negatively curved symmetric spaces whose isometry group does not contain a convex cocompact subgroup with limit set a Pontryagin sphere; these are precisely $\hh^2$, $\hh^3$, and $\hh_\cc^2$.
We also note that Granier \cite{Granier} constructed convex cocompact subgroups of $\mathrm{Isom}(\hh_\cc^2)$ with limit set a Menger curve (see also recent related work of Ma--Xie \cite{MaXie}).
Outside of the world of symmetric spaces, one can use Ontaneda's Riemannian hyperbolization  \cite{ON20} to construct many pinched Hadamard manifolds in any dimension $d\geq 4$, whose isometry groups admit convex cocompact subgroups with limit set a Pontryagin sphere.
We also remark that there exist Gromov hyperbolic groups with Gromov boundary a Pontryagin sphere, which do not admit any proper action on a real or complex hyperbolic space \cite{MR26}.

\subsection*{Acknowledgements} 
We thank Thierry Barbot, Nikolay Bogachev, Bena Tshishiku, Genevieve Walsh, and Theodore Weisman for useful conversations. We thank Ara Bamsajian, Jiming Ma, and Subhadip Dey for pointing us to relevant references. We also thank Theodore Weisman for providing the beautiful Figure~\ref{fig:pontryagin}.
S.D. thanks Fran\c{c}ois Gu\'eritaud for the content of Remark~\ref{remark:bourdon}. Finally, we would like to thank the referee for carefully reading the paper and for suggesting several valuable improvements.

S.D. was supported by the Huawei Young Talents Program.  G.L. was supported by the European Research Council under ERC-Consolidator Grant 614733 and by the National Research Foundation of Korea (NRF) grant funded by the Korean government (MSIT) (No 2020R1C1C1A01013667). L.M. acknowledges support by the Centre Henri Lebesgue (ANR-11-LABX-0020 LEBESGUE),  ANR G\'eom\'etries de Hilbert sur tout corps valu\'e (ANR-23-CE40-0012) and  ANR Groupes Op\'erant sur des FRactales (ANR-22-CE40-0004). L.R. acknowledges support by INDAM-GNSAGA.


\section{Preliminaries}

\subsection{Simplicial complexes}
Let $K$ be a  simplicial complex. 
In this paper $K$ will always be finite.
We will use the following terminology and notation:
\begin{itemize}
    \item The $d$-\textit{skeleton} of $K$ is denoted $K^{(d)}$.

    \item $K$ is \textit{flag} if any collection of $d+1$ pairwise adjacent vertices spans a $d$-simplex.
    If $K$ is any simplicial complex, then the \textit{flag complex} on $K$ is the complex obtained by adding all simplices whose $1$-skeleton appears in $K$.
    In particular, a flag complex is completely determined by its $1$-skeleton.

    \item If $L\subseteq K$ is a subcomplex, we say that $L$ is \textit{full} (or \textit{induced}) if 
    whenever $d+1$ vertices of $L$ span a $d$-simplex in $K$, they also span a $d$-simplex in $L$.
    In particular, note that a full subcomplex of a flag complex is flag, and that if $H$ is a full subcomplex of $L$ and $L$ is a full subcomplex of $K$, then $H$ is a full subcomplex of $K$.
    
    \item $K$ is \textit{flag-no-square} if it is flag and has no induced squares (i.e., every square has a chord).
    
    \item The \textit{puncture-respecting cohomological dimension} of $K$ is defined to be: 
$$\pcd(K)=\max_{\sigma} \,\, \max \left\lbrace n \, \left | \,\,   \overline H^n(K\setminus \sigma) \neq 0 \right . \right\rbrace,$$
where $\overline H^n$ denotes reduced cohomology and $\sigma$ is a (possibly empty) simplex of $K$; see \cite{DS21}.
        
\end{itemize}

\subsection{The Pontryagin sphere}\label{sec:pontryagin}
The \textit{Pontryagin sphere} is a $2$-dimensional continuum (i.e., connected compact metrizable space), which arises as the limit of a certain inverse system of closed surfaces of positive genus, see Figure~\ref{fig:pontryagin}. 
For more details and background about the material presented in this section, we refer the reader to \cite{JA91,FI03,ZA10,SW20,SW20sp}.
The name ``sphere'' can be misleading here. 
Indeed, not only is this space not a manifold, but also none of its nonempty open subsets can be embedded in  $\rr^2$. 
Note that Pontryagin also defined an infinite family of $2$-dimensional continua called \textit{Pontryagin surfaces}, which are also obtained as inverse limits of $2$-dimensional complexes but are not manifolds. 
However, the Pontryagin sphere is not one of the Pontryagin surfaces: while the former embeds in $\rr^3$ and has rational cohomological dimension $2$, the latter do not, see \cite[Example 1.9]{DR05}.

The notion of dimension we consider in this paper is the \textit{Lebesgue covering dimension} (or \textit{topological dimension}), i.e., the minimum integer $d$ for which every open cover has an open refinement such that every point is contained in at most $d+1$ open sets.
This is a well defined topological invariant for all metric spaces, and agrees with the standard notion of dimension for manifolds and CW-complexes.

Here are some more details about the construction of the Pontryagin sphere.
Let $\{f_{n}:S_n\to S_{n-1}\}_{n\in \nn}$ be an inverse system of closed connected  orientable surfaces and continuous maps between them.
Its inverse limit is called a \textit{tree of surfaces} if the bonding maps $f_n$ satisfy some technical conditions, which roughly speaking say that $S_n$ is obtained from $S_{n-1}$ by performing a connected sum with some closed orientable surfaces, in such a way that the portion of $S_{n-1}$ which is engaged in the connected sum becomes dense eventually.
The reader should imagine starting with a surface $S_0$ and performing an infinite sequence of connected sums with surfaces along smaller and smaller disks that become more and more densely packed.
If all surfaces involved are $2$-spheres, then the limit is a $2$-sphere. 
If finitely many surfaces have positive genus, then the limit is a closed surface of positive genus. 
If infinitely many surfaces have positive genus, then the limit is the Pontryagin sphere. Heuristically, this space feels like a $2$-sphere with infinitely many handles attached, but since we are attaching handles all over the place, the result is a fractal space rather than a surface of infinite type, see Figure~\ref{fig:pontryagin}.

In our context, an inverse system as above arises as follows.
Let $X$ be a CAT(0) cube complex in which the link of every vertex is a flag triangulation of a closed connected  orientable surface of positive genus.
Choose a basepoint $p$ in the interior of a cube and consider a family of closed metric balls $B(p,r_n)$ of radius $r_n>0$ centered at $p$. 
Assume for simplicity that $p$ and $r_n$ are chosen generically, so that no ball contains a vertex on its boundary. 
If $B(p,r_n)$ does not contain any vertex, then $\partial B(p,r_n)$ is just the space of directions at $p$ and is homeomorphic to a $2$-sphere. 
More generally, as observed in \cite[\S 3.d]{DJ91},  $\partial B(p,r_{n})$ is homeomorphic to the surface obtained from $\partial B(p,r_{n-1})$ by taking a connected sum with the surfaces $\lk{v}{X}$, where $v$ ranges over the  vertices contained in $B(p,r_{n})\setminus B(p,r_{n-1})$. 
We obtain an inverse system of surfaces, in which the bonding maps are induced by the nearest-point projections $B(p,r_{n})\to B(p,r_{n-1})$.
Note that the collection of geodesic rays issuing from the base point $p$ and hitting a vertex is dense in the space of directions at $p$, so the connected sums involved in the inverse limit engage a dense portion of the space of directions at $p$.
The \textit{visual boundary} $\partial_\infty X$ of $X$ is the limit of this inverse system, and is homeomorphic to the Pontryagin sphere.

\subsection{Right-angled Coxeter groups}
Given a flag complex $K$, the \textit{right-angled Coxeter group} (RACG) defined by $K$ is the group $W_K$ generated by one involution for each vertex of $K$, with two generators commuting exactly when the corresponding vertices are adjacent. In other words, it is defined by the following presentation
$$
W_K=\langle s\in K^{(0)} \mid s^2=1, \; [s,t]=1 \textrm{ for all edges } (s,t) \textrm{ of }  K \rangle.
$$

The RACGs considered in this paper arise  concretely as groups generated by reflections in the codimension-$1$ faces of certain polytopes in real hyperbolic $d$-space $\hh^d$. 
More generally, a RACG $W_K$ can be regarded as a group generated by reflections in the hyperplanes of the associated \textit{Davis complex}. 
This is a CAT(0) cube complex whose $1$-skeleton may be identified with the Cayley graph of $W_K$ and where the link of each vertex is isomorphic to the flag complex $K$.
For background on Coxeter groups the reader is referred to  \cite{DA08}.

The flag complex $K$ is called the \textit{nerve} of the group $W_K$. 
Note that $W_K$ is completely determined by the $1$-skeleton of $K$. 
So, if $\mathscr{G}$ is a finite simplicial graph, we also use the notation $W_{\mathscr{G}}$ for the RACG $W_K$, where $K$ is the flag complex on $\mathscr{G}$.

We now collect some well-known properties of $W_K$ that can be read off of its nerve for future reference.
Given a full subcomplex  $L\subseteq K$, the subgroup  $\langle L^{(0)} \rangle$ generated by the vertices of $L$ is called the \textit{special subgroup} defined by $L$.
\begin{lemma}\label{lem:racg}
    Let $K$ be a flag complex. Then the following hold.
    \begin{enumerate}
    \item \label{item:racg special} If $L\subseteq K$ is a  full subcomplex, then $\langle L^{(0)}\rangle$ is naturally isomorphic to $W_L$ and is quasiconvex (with respect to the generating set $K^{(0)}$ for $W_K$).
    
    \item \label{item:racg hyp} $W_K$ is Gromov hyperbolic if and only if $K$ is flag-no-square.

    \item \label{item:racg dim} If $K$ is flag-no-square, then  $\dim(\grbd W_K)=\pcd(K)$, where $\grbd W_K$ is the Gromov boundary of $W_K$.
    \end{enumerate}
\end{lemma}
\begin{proof}
The properties of special subgroups in \eqref{item:racg special} are classical, see e.g. \cite[\S 4.1]{DA08}. 
    Item \eqref{item:racg hyp} is the right-angled case of Moussong's theorem \cite{moussong}, see also \cite[Corollary 12.6.3]{DA08}.
    The equality in \eqref{item:racg dim} follows from  \cite[Corollary 1.4.(e)]{BM91} and \cite[Corollary 8.5.5]{DA08}.
\end{proof}

In many cases of interest, the topology of the nerve determines more than just the dimension of the Gromov boundary. 
For instance, we will use the following result from \cite{FI03}.

\begin{lemma}\label{lem:pontryagin}
 Let $K$ be a flag-no-square triangulation of a closed connected orientable surface of genus $g\geq 1$.   
 Then the Gromov boundary of $W_K$ is homeomorphic to the Pontryagin sphere.
\end{lemma}


 \subsection{Convex cocompactness}

Let $\Gamma$ be a discrete subgroup of $\mathrm{Isom} (\hh^d)$. 
The \textit{limit set} of $\Gamma$ is the closed subset of $\Sb^{d-1} = \partial_\infty \hh^d$ consisting of all accumulation points of the orbit under $\Gamma$ of any point of $\hh^d$ (this does not depend on the point).
We say that $\Gamma$ is \textit{convex cocompact} if $\Gamma$ acts cocompactly on the convex hull of its limit set.

\begin{lemma}\label{lem:convexccpt}
Let $\Gamma$ be a convex cocompact subgroup of $\mathrm{Isom} (\hh^d)$.     Then the following hold.
\begin{enumerate}
    \item \label{item:cc hyp} $\Gamma$ is Gromov hyperbolic, and the limit set of $\Gamma$ in $\Sb^{d-1}$ is homeomorphic to the Gromov boundary of~$\Gamma$.

    \item \label{item:cc sub} If $H\leq \Gamma$ is quasiconvex, then $H$ is convex cocompact.
\end{enumerate}    
\end{lemma}
\begin{proof}
Part \eqref{item:cc hyp} is  \cite[Theorem 12]{SW01}.
    Part \eqref{item:cc sub} follows from the fact that an $H$-orbit  is quasiconvex in the corresponding $\Gamma$-orbit, together with the characterization of convex cocompactness in terms of quasiconvexity of orbits from \cite[Main Theorem]{SW01}.
\end{proof}

\section{Construction in \texorpdfstring{$\hh^4$}{H4}}

\subsection{The 600-cell as a union of two solid tori}

Among the six convex regular $4$-polytopes, there is precisely one with $600$ facets, namely, the hexacosichoron.
From now on, we will refer to the simplicial complex given by the boundary of the hexacosichoron as the $600$-cell, denoted $\six$. 
As such, $\six$ is a triangulation of the 3-dimensional sphere $\Sb^3$ consisting of 600 tetrahedra, 1200 triangles, 720 edges, and 120 vertices. Each vertex is incident to 20 tetrahedra and 12 edges.
We will now review a certain decomposition of $\six$ that is relevant for the purposes of this paper.
For additional details about the 600-cell see \cite[pp. 19-23, \S9]{CO70}.

We may decompose $\six$ as a union of two solid tori  $\tbtm$  (the \textit{bottom torus}), $\ttop$ (the \textit{top torus}), and an \textit{interface}~$\interface$ homeomorphic to $T^2\times [0,1]$.
Each solid torus can be constructed as follows:
\begin{enumerate}
    
    \item A \textit{flying saucer} is the 3-dimensional simplicial complex obtained by gluing 5 tetrahedra along a common edge. 
    Equivalently, this is obtained by taking the suspension of the boundary of a pentagon and connecting the suspension vertices with an additional {\em internal} edge, and then filling with tetrahedra.
    The boundary is a pentagonal bipyramid.
    See Figure~\ref{fig:saucers}.

    \item Consider the complex obtained by stacking 10 flying saucers vertex to vertex in such a way that the internal edges form a cycle of length 10. Now, fill in the gaps between two consecutive flying saucers with 10 \textit{interstitial} tetrahedra so that the union of two consecutive flying saucers with the interstitial tetrahedra is an icosahedral pyramid. See Figure~\ref{fig:saucers}.
    The resulting complex is a triangulation of a solid torus $T_i$ consisting of 150 tetrahedra.

    \item The surface $\partial T_i$ of each of  $T_i$ is the same as the outer surface of the complex consisting of 10 pentagonal antiprisms---so-called {\em drums}---stacked together to form a torus with 100 exposed\footnote{A simplex is {\em exposed} if it is contained in $\partial T_i$.} triangles  (with a contribution of $10$ from each drum), 150 exposed edges, and 50 exposed vertices. 
    In particular the triangulation of $\partial T_i$ is that of a torus by 100 triangles, in which every vertex has 6 neighbors. See Figure~\ref{fig:h4}.
\end{enumerate}

\begin{figure}[h]
\includegraphics[width=.75\columnwidth]{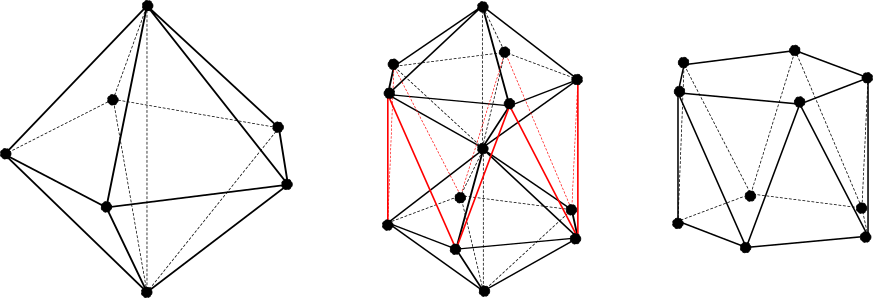}
\caption{From left to right: a flying saucer, two flying saucers stacked vertex-to-vertex with ten interstitial tetrahedra (indicated in red) around their vertex of intersection, and a drum.}
\label{fig:saucers}
\end{figure}

The interface $\interface$ has the following structure:
\begin{enumerate}
    \item To each of the 100 exposed triangles of $\tbtm$, we glue a tetrahedron (which we call a \textit{raised} tetrahedron), in the following way: the raised tetrahedra come in pairs, identified along the triangular faces that share an edge making an angle of $60^\circ$ with the positive horizontal direction in Figure~\ref{fig:h4}. 
    The unique common vertex among such a pair of tetrahedra which is not a vertex of $\tbtm$ will be called the \textit{raised vertex} of that pair. There are 50 raised vertices.
    We do the same for $\ttop$.
    
    \item The interface is obtained by interlocking the raised tetrahedra for the top and bottom tori, so that the raised vertex of a raised tetrahedron on $\tbtm$ gets identified with an exposed vertex in $\ttop$, and vice versa. 
    After doing so, some gaps are left, that can be filled with  100 \textit{filling} tetrahedra, each of which has an edge on $\tbtm$ and an edge on $\ttop$.
\end{enumerate}

Let $v$ be a vertex on the surface of one of the $T_i$, say, the bottom torus $\tbtm$. 
We now describe the 12 neighbors of $v$ in $\six$ with respect to the above decomposition of $\six$:
\begin{enumerate}

\item 2 neighbors  in the interior of $\tbtm$. 

\item 6 neighbors $w_i$ on $\partial \tbtm$. 
The edges $[v,w_i]$ are on $\partial \tbtm$.

\item 4 neighbors  on $\partial \ttop$; indeed $v$ is a raised vertex for $\ttop$, and is hence the raised vertex for a pair of raised tetrahedra of $\ttop$.
\end{enumerate}


\subsection{Proof of Theorem~\ref{main4}}

 Among the six convex regular $4$-polytopes, there is precisely one with 120 facets, namely,
 the dodecacontachoron. The right-angled dodecacontachoron can be realized as a compact polytope $\mathcal{P}_{120}$ in $\hh^4$ (see e.g. \cite{Coxeter56,DA85} or \cite[Section~1.3]{martelli_at_rescue}).
 By Poincar\'e's polyhedron theorem, the group $\mathrm{Refl}(\mathcal{P}_{120})$ generated by the reflections in the facets of $\mathcal{P}_{120}$ is a uniform lattice in $\hh^4$ and $\mathrm{Refl}(\mathcal{P}_{120})$ is the RACG defined by the 1-skeleton of the dual of $\partial \mathcal{P}_{120}$, i.e., the 1-skeleton of $\six$, so that $\mathrm{Refl}(\mathcal{P}_{120})\cong W_{\six}$.

Being isomorphic to a uniform lattice of $\mathrm{Isom} (\hh^4)$, the group $W_{\six}$ is a Gromov hyperbolic group with Gromov boundary $\Sb^3$. In particular, $\six$ is a flag-no-square complex, by Lemma~\ref{lem:racg}.\eqref{item:racg hyp}. We will show that the subgroup $\Gamma^4$ of $W_{\six}$ generated by the vertices in $\partial \tbtm$ (or, equivalently, $\partial \ttop$) has limit set a Pontryagin sphere. The triangulation of $\partial \tbtm$ is as in Figure~\ref{fig:h4}, with opposite edges are identified.
\begin{figure}[ht]
\centering
\def\svgwidth{.75\columnwidth}
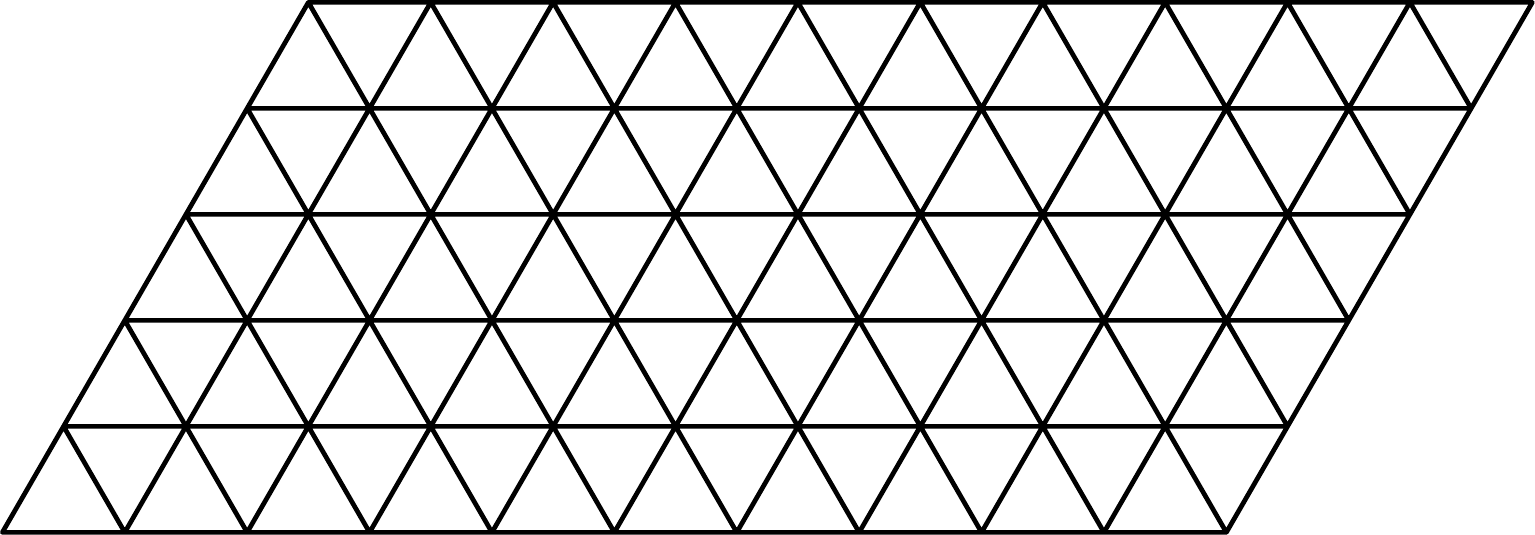  
\caption{The triangulation of $\partial \tbtm$ with 50 vertices, 150 edges, and 100 triangles. Opposite edges are identified. This defines a RACG generated by 50 reflections inside $W_{\six}$.} 
\label{fig:h4}
\end{figure}

\begin{proposition}
The subcomplex $\partial \tbtm$ of $\six$ is a flag-no-square triangulation of a torus and a full subcomplex of $\six$.
\end{proposition}

\begin{proof}
The triangulation is clearly flag-no-square; indeed this is a quotient of an equilateral triangulation of a parallelogram in $\rr^2$, whose sides have combinatorial length larger than $4$, see Figure~\ref{fig:h4}.
To see that $\partial \tbtm$ is a full subcomplex we argue as follows.
Let $v$ be a vertex of $\partial \tbtm$. 
The above description of the neighbors of $v$ shows that all the edges of $\six$ connecting $v$ with another vertex on $\partial \tbtm$ are already contained in $\partial \tbtm$, namely they are the 6 edges to the 6 neighbors on $\partial \tbtm$. 
So, the 1-skeleton of $\partial \tbtm$ is an induced subgraph of the 1-skeleton of $\six$.
Since $\partial \tbtm$ is a flag 2-complex, it follows that $\partial \tbtm$ is a full subcomplex.
\end{proof}

With this proposition, we can see that the limit set of $\Gamma^4$ is a Pontryagin sphere, as follows. 
Since $\partial \tbtm$ is a full subcomplex of $\six$, 
it follows from \eqref{item:racg special} in Lemma~\ref{lem:racg}
that the special subgroup $\Gamma^4$ is quasiconvex and isomorphic to the abstract RACG $W_{\partial \tbtm}$.
Since $\mathrm{Refl}(\mathcal{P}_{120})$ acts cocompactly on $\hh^4$, Lemma~\ref{lem:convexccpt} implies that~$\Gamma^4$ is convex cocompact in $\mathrm{Isom}(\hh^4)$ and in particular that the limit set of $\Gamma^4$ is homeomorphic to the Gromov boundary of $\Gamma^4$.
Finally,  since the nerve $\partial \tbtm$ of $W_{\partial \tbtm}$ is a flag-no-square triangulation of a closed orientable surface of positive genus, the Gromov boundary of $\Gamma^4 \cong W_{\partial \tbtm}$ is a Pontryagin sphere by Lemma~\ref{lem:pontryagin}.

\begin{remark}
    By construction, $\Gamma^4$ is a subgroup of a uniform lattice $\mathrm{Refl}(\mathcal{P}_{120})$ in $\hh^4$. It follows from work of Bugaenko \cite{bugaenko} that $\mathrm{Refl}(\mathcal{P}_{120})$ has finite index in the group $O(f_\varphi , \mathbb{Z}[\varphi])$, where $\varphi = \frac{1+ \sqrt{5}}{2}$, and $f_\varphi$ is the bilinear form on $\mathbb{R}^5$ given by the diagonal matrix $(1,1,1,1, -\varphi)$, so that $\mathrm{Refl}(\mathcal{P}_{120})$ is an arithmetic reflection lattice.
\end{remark}

\begin{remark}
Figure~\ref{fig:pontryagin} was obtained by Theodore Weisman by plotting the limit set of the action of $\Gamma^4$ on $\hh^4$ given by restricting the action of $\mathrm{Refl}(\mathcal{P}_{120})$.
\end{remark}


\section{Construction in \texorpdfstring{$\hh^6$}{H6}}

\subsection{An abstract Coxeter group with \texorpdfstring{$21$}{21} generators}

Let $\T$ be the flag simplicial complex determined by the graph depicted in Figure \ref{fig:nerve} (vertices sharing the same label should be identified and the boundary edges should be identified accordingly). Since every homotopically nontrivial edge-cycle in $\T$ has combinatorial length greater than $4$, the complex is flag-no-square. The geometric realization of $\T$ is a topological surface $S$ with 21 vertices, 63 edges and 42 faces, so that $\chi (S)= 0$. Hence, $\T$ is a flag-no-square triangulation of a 2-torus. The goal is to build a convex cocompact action of $W_{\T}$ on $\hh^6$ as a reflection group.

In fact, the presentation of $W_{\T}$ can be expressed rather simply, and usefully, as follows. For $i=1$, the vertices adjacent to the vertex labeled $i$ in $\T$ have the labels $i \pm 1$, $i \pm 4$, and $i \pm 5$ (mod $21$). Additionally, when one goes from the vertex $i$ to the vertex $i+1$, the labels of all adjacent vertices increase by $1$ (mod $21$). Hence, let $M = (m_{i,j})$ be the symmetric $21 \times 21$ matrix with entries in $\{ 1, 2, \infty \}$ given by
\begin{equation*}
m_{i,j} =
\begin{cases}
1 & \text{if } i = j;\\
2 & \text{if } |i - j| = \pm 1, \pm 4 \text{ or } \pm 5 \; (\bmod\, 21);\\
\infty & \text{otherwise}.
\end{cases}
\end{equation*}
Then
\begin{equation*}
W_{\T} = \langle s_1, \dotsc, s_{21} \mid (s_i s_j)^{m_{i,j}} = 1, \;\forall 1 \leqslant i,j \leqslant 21 \rangle.
\end{equation*}
Since $\T$ is flag-no-square, we have by Lemma~\ref{lem:racg}.\eqref{item:racg hyp} that $W_{\T}$ is Gromov hyperbolic.

\begin{figure}[ht!]
\begin{tikzpicture}[thick,scale=0.8, every 
node/.style={transform shape}]

\node[draw,circle, inner sep=3.5pt, minimum size=3.5pt] (15B) at (-1.3,2.25) {{8}};
\node[draw,circle, inner sep=3.5pt, minimum size=3.5pt] (11B) at (-1.3,3.75) {{9}};

\node[draw,circle, inner sep=2pt, minimum size=2pt] (3B) at (0,0) {{11}};
\node[draw,circle, inner sep=2pt, minimum size=2pt] (20B) at (0,1.5) {{12}};
\node[draw,circle, inner sep=2pt, minimum size=2pt] (16) at (0,3.0) {{13}};
\node[draw,circle, inner sep=2pt, minimum size=2pt] (12B) at (0,4.5) {{14}};
\node[draw,circle, inner sep=2pt, minimum size=2pt] (8C) at (0,6.0) {{15}};

\node[draw,circle, inner sep=2pt, minimum size=2pt] (8B) at (1.3,-0.75) {{15}};
\node[draw,circle, inner sep=2pt, minimum size=2pt] (4) at (1.3,0.75) {{16}};
\node[draw,circle, inner sep=2pt, minimum size=2pt] (21) at (1.3,2.25) {{17}};
\node[draw,circle, inner sep=2pt, minimum size=2pt] (17) at (1.3,3.75) {{18}};
\node[draw,circle, inner sep=2pt, minimum size=2pt] (13) at (1.3,5.25) {{19}};
\node[draw,circle, inner sep=2pt, minimum size=2pt] (9B) at (1.3,6.75) {{20}};

\node[draw,circle, inner sep=2pt, minimum size=2pt] (9A) at (2.6,0) {{20}};
\node[draw,circle, inner sep=2pt, minimum size=2pt] (5) at (2.6,1.5) {{21}};
\node[draw,circle, inner sep=3.5pt, minimum size=3.5pt] (1) at (2.6,3.0) {{1}};
\node[draw,circle, inner sep=3.5pt, minimum size=3.5pt] (18) at (2.6,4.5) {{2}};
\node[draw,circle, inner sep=3.5pt, minimum size=3.5pt] (14A) at (2.6,6.0) {{3}};

\node[draw,circle, inner sep=3.5pt, minimum size=3.5pt] (14B) at (3.9,-0.75) {{3}};
\node[draw,circle, inner sep=3.5pt, minimum size=3.5pt] (10) at (3.9,0.75) {{4}};
\node[draw,circle, inner sep=3.5pt, minimum size=3.5pt] (6) at (3.9,2.25) {{5}};
\node[draw,circle, inner sep=3.5pt, minimum size=3.5pt] (2) at (3.9,3.75) {{6}};
\node[draw,circle, inner sep=3.5pt, minimum size=3.5pt] (19) at (3.9,5.25) {{7}};
\node[draw,circle, inner sep=3.5pt, minimum size=3.5pt] (15A) at (3.9,6.75) {{8}};

\node[draw,circle, inner sep=3.5pt, minimum size=3.5pt] (15C) at (5.2,0) {{8}};
\node[draw,circle, inner sep=3.5pt, minimum size=3.5pt] (11A) at (5.2,1.5) {{9}};
\node[draw,circle, inner sep=2pt, minimum size=2pt] (7) at (5.2,3.0) {{10}};
\node[draw,circle, inner sep=2pt, minimum size=2pt] (3A) at (5.2,4.5) {{11}};
\node[draw,circle, inner sep=2pt, minimum size=2pt] (20A) at (5.2,6.0) {{12}};

\node[draw,circle, inner sep=2pt, minimum size=2pt] (12A) at (6.5,2.25) {{14}};
\node[draw,circle, inner sep=2pt, minimum size=2pt] (8A) at (6.5,3.75) {{15}};

\draw (15B) -- (11B);

\draw (3B) -- (20B);
\draw (20B) -- (16);
\draw (16) -- (12B);
\draw (12B) -- (8C);

\draw (8B) -- (4);
\draw (4) -- (21);
\draw (21) -- (17);
\draw (17) -- (13);
\draw (13) -- (9B);

\draw (9A) -- (5);
\draw (5) -- (1);
\draw (1) -- (18);
\draw (18) -- (14A);

\draw (14B) -- (10);
\draw (10) -- (6);
\draw (6) -- (2);
\draw (2) -- (19);
\draw (19) -- (15A);

\draw (15C) -- (11A);
\draw (11A) -- (7);
\draw (7) -- (3A);
\draw (3A) -- (20A);

\draw (12A) -- (8A);
\draw (8C) -- (9B);

\draw (11B) -- (12B);
\draw (12B) -- (13);
\draw (13) -- (14A);
\draw (14A) -- (15A);

\draw (15B) -- (16);
\draw (16) -- (17);
\draw (17) -- (18);
\draw (18) -- (19);
\draw (19) -- (20A);

\draw (20B) -- (21);
\draw (21) -- (1);
\draw (1) -- (2);
\draw (2) -- (3A);

\draw (3B) -- (4);
\draw (4) -- (5);
\draw (5) -- (6);
\draw (6) -- (7);
\draw (7) -- (8A);

\draw (8B) -- (9A);
\draw (9A) -- (10);
\draw (10) -- (11A);
\draw (11A) -- (12A);

\draw (14B) -- (15C);

\draw (15A) -- (20A);

\draw (9B) -- (14A);
\draw (14A) -- (19);
\draw (19) -- (3A);
\draw (3A) -- (8A);

\draw (8C) -- (13);
\draw (13) -- (18);
\draw (18) -- (2);
\draw (2) -- (7);
\draw (7) -- (12A);

\draw (12B) -- (17);
\draw (17) -- (1);
\draw (1) -- (6);
\draw (6) -- (11A);

\draw (11B) -- (16);
\draw (16) -- (21);
\draw (21) -- (5);
\draw (5) -- (10);
\draw (10) -- (15C);

\draw (15B) -- (20B);
\draw (20B) -- (4);
\draw (4) -- (9A);
\draw (9A) -- (14B);

\draw (3B) -- (8B);
\end{tikzpicture}
\caption{The nerve of the Coxeter group $W_{\T}$}\label{fig:nerve}
\end{figure}
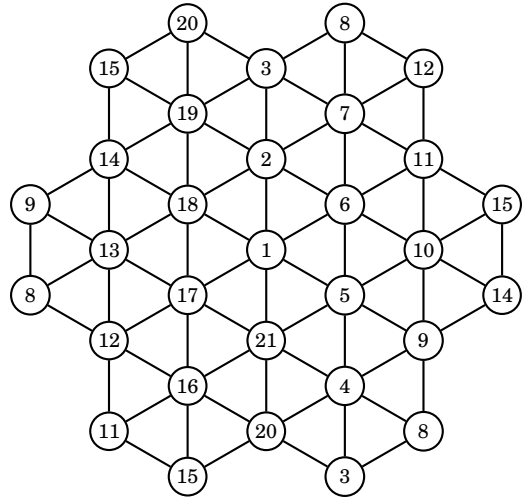

\subsection{Construction of the reflection group \texorpdfstring{$\Gamma^6$}{Gamma6}  of Theorem~\ref{main6}}

In this section, we construct an action $W_{\T}$ on $\hh^6$ as a reflection group.

\begin{proposition}\label{main6_1st_half}
 There exists a convex cocompact right-angled reflection group $\Gamma^6$ of $\mathrm{Isom} (\hh^6)$ whose limit set is homeomorphic to the Pontryagin sphere.
\end{proposition}

\begin{proof}
Let $A = (A_{i,j})$ be the symmetric $21 \times 21$ matrix given by 

\begin{equation*}
A_{i,j} =
\begin{cases}
1 & \text{if } i = j; \\
0 & \text{if } |i - j| = \pm 1, \pm 4 \text{ or} \pm 5 \; (\bmod\, 21);\\
-u & \text{if } |i - j| = \pm 3, \pm 6 \text{ or} \pm 9 \; (\bmod\, 21);\\
-v & \text{if } |i - j| = \pm 2, \pm 8 \text{ or} \pm 10 \; (\bmod\, 21);\\
-w & \text{if } |i - j| = \pm 7 \; (\bmod\, 21),
\end{cases}
\end{equation*}
where $u= \frac{1}{50} \left(27+7\sqrt{21}\right)$, $v = \frac{1}{50} \left(21+11 \sqrt{21}\right)$ and $w = \frac{1}{50} \left(49+9 \sqrt{21}\right)$.  
A computation shows that the matrix $A$ has rank $7$ and signature $(6,1,14)$, see Remark~\ref{rem:computation}.
So, \cite[Theorem~2.1]{Vin85} shows that there exists a right-angled Coxeter polytope $\mathcal{P}_{21}$ with $21$ facets in $\hh^6$ whose Gram matrix is $A$. The group $\Gamma^6 =  \mathrm{Refl}(\mathcal{P}_{21})$ generated by the $21$ reflections along the boundary hyperplanes of $\mathcal{P}_{21}$ is isomorphic to $W_{\T}$, by Poincar\'e's polyhedron theorem, since $u, v, w \geq 1$.

The group $\Gamma^6$ is convex cocompact since $\mathcal{P}_{21}$ has no asymptotic facets, as $u, v, w > 1$ (see \cite[Theorem~4.12]{DH13}). Lemma~\ref{lem:convexccpt}.\eqref{item:cc hyp} shows that the Gromov boundary of $W_{\T}$ is homeomorphic to the limit set of $\Gamma^6$. As the nerve of $W_{\T}$ is a flag-no-square triangulation of a torus, the Gromov boundary of $W_{\T}$ is a Pontryagin sphere, by Lemma~\ref{lem:pontryagin}. Hence, the limit set of the convex cocompact subgroup $\Gamma^6$ is a Pontryagin sphere.
\end{proof}

\begin{remark}\label{rem:computation}
The above values for  $u,v,w$ were obtained as follows. 
We initialized a $21\times 21$ matrix $A$ with variable entries $u,v,w$, and looked for values for $u,v,w$ at which $A$ has rank $7$ and signature $(6,1,14)$. 
To this end, we sampled some $8\times 8$ minors and imposed that their determinant be zero. 
This provided a system of polynomial equations in $u,v,w$. First, we solved these numerically, then the LLL method was used to obtain algebraic numbers of bounded degree close to the numerical solutions. 
These are the values for $u,v,w$ written above, and we verified a posteriori that the matrix obtained for those specific values had the desired signature.
This was done by an exact algebraic computation  with SageMath. 
Detailed computations are provided in the SageMath file (\emph{computations}.ipynb), available on the webpage \cite{computations}.
The fact that only three variables sufficed is of course surprising, but probably is a consequence of the fact that $\T$ is very symmetric.
\end{remark}

\subsection{A two generator group that contains \texorpdfstring{$\Gamma^6$}{Gamma6}  as a finite-index subgroup}\label{sec:two_generators}

\begin{proposition}\label{main6_2nd_half}
  There exists a 2-generated convex cocompact subgroup $\Gamma_2^6$ of $\mathrm{Isom} (\hh^6)$ containing $\Gamma^6$ as a normal subgroup of index $21$. 
\end{proposition}

\begin{proof}

Let $$ R_m = \left(\begin{array}{rr}
\cos\left(\frac{2m\pi}{21}\right) & -\sin\left(\frac{2m\pi}{21}\right) \\
\sin\left(\frac{2m\pi}{21}\right) & \cos\left(\frac{2m\pi}{21}\right) \\
\end{array}\right)
\quad \textrm{and} \quad
\sigma = \left( \begin{array}{rrrr} 
R_1 & 0 & 0 & 0\\
0 & R_4 & 0 & 0 \\
0 & 0 & R_5 & 0\\
0 & 0 & 0 & 1 \\
\end{array} \right) \in \mathrm{SO}(6,1).$$

Let $y = \left( \alpha, 0 , \alpha, 0, \alpha, 0, \beta  \right)$, where $\alpha = \tfrac{1}{5}\sqrt{11 + \sqrt{21}}$ and $\beta = \tfrac{1}{5}\sqrt{8 + 3\sqrt{21}}$. Note that $y$ is a space-like vector of $\mathbb{R}^{6,1}$ with length $1$, that is,  $\langle y , y \rangle = 1$, where $\langle \phantom{y}, \phantom{y}  \rangle$ is the standard Minkowski metric on $\rr^{6,1}$. Therefore, for each integer $k \in \{ 1, 2, \dotsc, 21 \}$, a reflection $s_k \in \mathrm{O}(6,1)$ can be defined by 
$$ s_k (x) = x - 2 \langle \sigma^{k-1}(y), x \rangle \sigma^{k-1}(y) \quad\textrm{for all } x \in \mathbb{R}^{6,1}.$$
Note that $s_k = \sigma^{k-1} s_1 \sigma^{-k+1}$ for all $k \in \{ 1, 2, \dotsc, 21 \}$ and we can compute that  $$\langle \sigma^{i-1}(y), \sigma^{j-1}(y) \rangle = A_{i,j} \quad\textrm{for all } i, j \in \{ 1, 2, \dotsc, 21 \}. $$
Again, detailed computations can be found in the SageMath file (\emph{computations}.ipynb), available on the webpage \cite{computations}.

Since the group $\Gamma^6$ and the group generated by $s_1, \ldots, s_{21}$ both lie in $\mathrm{O}(6,1)$ and share the same Gram matrix, they are conjugate by \cite[Theorem~2.1]{Vin85}. Let $\Gamma_2^6$ be the group generated by $\sigma$ and $s_1$. From the relations, for $k \in \{ 1, 2, \dotsc, 21 \}$, $\sigma^{k-1} s_1 \sigma^{-k+1} = s_k$, we get that $\Gamma^6$ is a normal subgroup of $\Gamma_2^6$. The index of $\Gamma^6$ in $\Gamma^6_2$ is equal to the order of $\sigma$ in the quotient $\Gamma_2^6/\Gamma^6$, hence this index is $1,3,7$ or $21$. If the order of $\sigma$ in $\Gamma_2^6/\Gamma^6$ is not $21$ then $\Gamma^6$ contains $\sigma$, $\sigma^3$, or $\sigma^7$, and hence contains an element of order $21$, $7$, or $3$. But, in a Coxeter group, a finite-order element is conjugate into a spherical special subgroup (see \cite[Theorem.~12.3.4.(i)]{DA08}). In particular, in a right-angled Coxeter group, finite-order elements are of order a power of $2$, and hence $\sigma$ must be of order $21$ in $\Gamma_2^6/\Gamma^6$.
\end{proof}

\begin{proof}[Proof of Theorem~\ref{main6}]
A finite-index subgroup of a convex cocompact group is convex cocompact as well, and has the same limit set.
So, Propositions~\ref{main6_1st_half} and \ref{main6_2nd_half}   together imply Theorem~\ref{main6}.
\end{proof}

\subsection{A few remarks}\label{sec:few_remarks}

\begin{proposition}
    The group $\Gamma^6$ is not contained, even virtually, in an arithmetic lattice in $\mathrm{Isom}(\hh^6)$.
\end{proposition}

\begin{remark}
    We do not know if $\Gamma^6$ is contained, virtually or not, in a non-arithmetic lattice in $\mathrm{Isom}(\hh^6)$.
\end{remark}

\begin{proof}
    The Gram matrix of $\mathcal{P}_{21}$ contains three principal submatrices that are very symmetric. More precisely, for $j=1,2,3$, let $S_j = \{ k \in \{ 1, ..., 21 \}  \,|\, k \equiv j \,\, (\mathrm{mod} \, 3)\}$. Then the principal submatrix $A_j$ of the Gram matrix $A$ obtained by deleting the rows and columns not in $S_j$ is equal to :
    $$
    A_1 = A_2 = A_3 = \left(\begin{array}{rrrrrrr}
1 & -u & -u & -u & -u & -u & -u \\
-u & 1 & -u & -u & -u & -u & -u \\
-u & -u & 1 & -u & -u & -u & -u \\
-u & -u & -u & 1 & -u & -u & -u \\
-u & -u & -u & -u & 1 & -u & -u \\
-u & -u & -u & -u & -u & 1 & -u \\
-u & -u & -u & -u & -u & -u & 1
\end{array}\right)
$$
Each $A_j$ has signature $(6,1,0)$.
Geometrically, the Coxeter polytope $\mathcal{P}_{21}$ is the intersection of three isometric polyhedra $\Delta_j$ of $\hh^6$, for $j=1,2,3$, corresponding to the submatrices $A_j$. Each $\Delta_j$ is the intersection with $\hh^6$ (in the projective model) of a simplex in $\mathbb{RP}^6$ all of whose facets have nonempty intersection with $\mathbb{H}^6$.

Let $f$ be  the bilinear form on $\rr^7$ defined by the symmetric matrix $A_1$ (sometimes called the \emph{Tits} bilinear form associated with $A_1$ in the context of reflection groups), and
let $\mathrm{Refl} (\Delta_1)$ be the  subgroup of $\mathrm{O}(f,\rr)$ generated by reflections across the hyperplanes orthogonal to the canonical basis vectors of $\rr^7$ with respect to the bilinear form $f$.
We claim that $\mathrm{Refl} (\Delta_1)$ is not virtually contained in an arithmetic lattice of $\mathrm{O}(f,\rr)$, which  implies that  $\mathrm{Refl} (\mathcal{P}_{21})$ is not virtually contained in an arithmetic lattice.    

To that end, note that any $\mathrm{Refl} (\Delta_1)$-invariant subspace of $\rr^7$  either is contained in the intersection of the reflection hyperplanes or contains the subspace spanned by their dual vectors (see e.g. \cite[Proposition~19]{Vin71} or \cite[Proposition~3.23]{CG_CC}). Since $A_1$ has full rank, any such subspace is necessarily trivial.  Hence, the action of $\mathrm{Refl} (\Delta_1)$ on $\rr^7$ is irreducible. Since $f$ is Lorentzian, any irreducible subgroup of $\mathrm{O}(f,\rr)$ has Zariski closure containing $\mathrm{SO}(f,\rr)$ (see e.g. \cite[Proposition~1]{benoist_harpe}). Therefore, $\mathrm{Refl} (\Delta_1)$ is Zariski-dense in $\mathrm{O}(f,\rr)$.

Now suppose for a contradiction that there exists a finite index subgroup $\Gamma_1$ of $\mathrm{Refl}(\Delta_1)$ that is contained in an arithmetic lattice $\Lambda$ of $\mathrm{O}(f,\rr)$. It is classical that, in the language of \cite{Vin71}, one obtains from the definition of an arithmetic lattice (see e.g. \cite[\S IX.1]{Margulis} and \cite[Theorem~2]{Vin71}) that the group $\mathrm{Ad}(\Lambda)$ is definable over $\mathcal{A}$, where $\mathcal{A}$ denotes the ring of real\footnote{The fact that $\mathscr{A}$ can be taken to be the ring of {\em real} algebraic integers, as opposed to the entire ring of algebraic integers, uses that $\mathrm{O}(f, \rr)$ is not locally isomorphic to a complex Lie group.
In fact, the arithmetic origin of all arithmetic lattices in $\mathrm{O}(2k,1)$ for $k \geq 1$, as well as all arithmetic lattices in $\mathrm{O}(n,1)$ containing Zariski-dense subgroups of $\mathrm{SO}(n-1,1)$ for any $n \geq 2$, is completely straightforward; such lattices are all of simplest type (see \cite{BBKS} and the references therein).} algebraic integers and $\mathrm{Ad}$ denotes the adjoint representation of $\mathrm{O}(f, \rr)$.
So $\mathrm{Ad}(\Gamma_1)$ is definable over $\mathcal{A}$ as well. 
It then follows from \cite[Theorem~3]{Vin71}  that $\mathrm{Ad}(\mathrm{Refl}(\Delta_1))$ is also definable over~$\mathcal{A}$. On the other hand, since each standard generator $s: \rr^7 \rightarrow \rr^7$ of $\mathrm{Refl}(\Delta_1)$ is given, using the standard basis vector $e_s \in \mathbb{R}^7$, by $v \mapsto v - (v^T A_1 e_s)$, we have that  $\mathrm{Refl}(\Delta_1) \subset \mathrm{O}(f,\qq(\sqrt{21})) \subset \mathrm{O}(f,\overline{\qq}\cap \mathbb{R})$, and hence we obtain again from \cite[Theorem~2]{Vin71} that $\mathrm{Refl}(\Delta_1)$ is itself definable over~$\mathcal{A}$. In particular, by \cite[Lem.~1]{Vin71}, we have that the trace of each element of $\mathrm{Refl}(\Delta_1)$ is an algebraic integer. 
However, the trace of the product of two generators of $\mathrm{Refl} (\Delta_1)$ is equal to $4u^2+3 = \frac{21}{625}(173 + 18 \sqrt{21})$, which is not an algebraic integer.
\end{proof}

\begin{remark}\label{rem:zariski_dense}
Since the subgroup $\mathrm{Refl}(\Delta_1)$ is Zariski-dense in $\mathrm{O}(f, \mathbb{R})$, we have that  the group $\Gamma^6$ is Zariski-dense in $\mathrm{O}(f, \rr)$ as well. 
Hence, the limit set of $\Gamma^6$ is not contained in any round $k$-sphere of $\mathbb{S}^k = \partial \mathbb{H}^{k+1}$, for $k \leqslant 4$.
\end{remark}


\section{Special subgroups with limit set a Menger curve}\label{sec:menger}

In the sequel, a simplical complex $K$ is \textit{non-planar} if $K$ does not embed in $\rr^2$. 
The purpose of this section is to prove the following proposition.

\begin{proposition}\label{prop:menger_surface}
    Let $K$ be a flag-no-square triangulation  of a  closed connected orientable surface $S$ of genus $g\geq 1$.
    Let $n\geq 1$, let $V=\{v_1,\dots,v_n\}\subseteq K^{(0)}$ be a collection of pairwise non-adjacent vertices, and let $L$ be the full subcomplex  spanned by $K^{(0)}\setminus V$. 
    Then the following hold.
    \begin{enumerate}
    \item \label{item:menger subcomplex}
    $L$ is a  non-planar inseparable full subcomplex with $\pcd (L)=1$.

    \item \label{item:menger subgroup} The special subgroup $W_L$ has Gromov boundary  homeomorphic to the Menger curve.
    \end{enumerate}
\end{proposition}

Our proof builds on the work in \cite{SW16,DS21,DHW23} and needs the following terminology.
We will say that $K$ is \textit{inseparable} if it is connected and has no separating simplex, no separating pair of non-adjacent vertices, and no separating suspension of a simplex. 
In other words, $K$ cannot be disconnected by removing any of the following \textit{forbidden subcomplexes}: a vertex, an edge, a triangle, two non-adjacent vertices, two edges meeting at a vertex, or two triangles meeting along an edge.

\begin{lemma}\label{lem:menger_vertex}
    Let ${K}$ be a flag-no-square triangulation  of a  closed connected orientable surface $S$ of genus $g\geq 1$.
    Let $M \subseteq {K}$ be a non-planar inseparable full subcomplex. 
    Let $v \in M^{(0)}$ such that $\lk{v}{M}$ is a circle,
    and let $L$ be the full subcomplex of $M$ spanned by $M^{(0)}\setminus \{v\}$.
    Then $L$ is a non-planar inseparable full subcomplex with $\pcd (L)=1$.
\end{lemma}
\begin{proof}
We start by noticing that both $M$ and $L$ are   flag-no-squares   complexes of dimension at most $2$, because they are full subcomplexes of $ K$. 
Since $M$ is non-planar and not a sphere, and $L$ is obtained from $M$ by removing the open disk bounded by the circle $\lk vM$, it follows that $L$ is non-planar as well.
 
To compute the cohomological dimension, note that for any (possibly empty) simplex $\sigma$ of $L$, the complex $L\setminus \sigma$ deformation retracts to a graph (it can be thickened to a subsurface with non-empty boundary of $S$).
In particular, the cohomology of $L\setminus \sigma$ always vanishes in degree $2$ and above.
On the other hand, note that $L$ does not deformation retract to a tree, because $L$ is non-planar. So we conclude that $\pcd(L)=1$.

To conclude, we show that $L$ is inseparable. 
Since $M$ is connected and $\lk vM$ is connected, we have that $L$ is connected. 
(This follows from Mayer--Vietoris applied to the decomposition  $M=\st vM \cup L$, for which $\st vM\cap L=\lk vM$.)  
Let $A\subseteq L$ be a subcomplex that disconnects $L$ but not $M$.
We need to check that~$A$ is not among the forbidden subcomplexes.
Since $\lk vM$ is connected, $A\cap \lk vM$ must disconnect $\lk vM$.
Since $\lk vM$ is a circle, $A\cap \lk vM$ must have at least $2$ non-adjacent vertices.
In particular, $A$ cannot be a simplex.
For the remaining cases we argue by contradiction.
Let $A=\{p,q\}\ast \sigma$ be the suspension of a simplex. Then $A\cap \lk vM$  consists of the two suspension points $\{p,q\}$. 
But then for any vertex $r\in \sigma$ we have that $p,q,r,v$ form an induced square in $M$, which cannot happen.
 Finally, let $A=\{u,w\}\subseteq \lk vM$ be a pair of non-adjacent vertices. 
 But then $\{u,v,w\}$ separates $M$.
 This is absurd since $M$ has no separating suspension of a simplex.
 This concludes the proof that $L$ is inseparable. 
\end{proof}

\begin{proof}[Proof of Proposition~\ref{prop:menger_surface}]
First of all, note that since $K$ is a simplicial triangulation of a surface, the star of any vertex is an embedded closed disk, hence the link is a circle.
Moreover, the open stars of the vertices in~$V$ are disjoint.

To prove \eqref{item:menger subcomplex} we argue by induction on $n$.
For $n=1$, the statement follows from Lemma~\ref{lem:menger_vertex} with $M = K$ and $v\in K^{(0)}$ any vertex.
Indeed,  $K$ is connected and non-planar.
Moreover, since $K$ triangulates a surface, it cannot be disconnected by any of the forbidden subcomplexes, so $K$ is inseparable.

For the inductive step, let $n\geq 2$ and let $V=\{v_1,\dots,v_n\}\subseteq K^{(0)}$ be a collection of pairwise non-adjacent vertices.
Let $L$ be the full subcomplex  spanned by $K^{(0)}\setminus V$.
By induction we know that the full subcomplex $M$ spanned by $K^{(0)} \setminus \{v_1,\dots,v_{n-1}\}$ is non-planar and inseparable.
For $i=1,\dots, n-1$ we have that $v_i$ is not adjacent to $v_n$, so  we have that 
$\lk {v_n}{M}=\lk {v_n}{K}$ is a circle.
 Since $L$ is the full subcomplex spanned by $M^{(0)}\setminus \{v_n\}$, the desired properties of $L$ follow from Lemma~\ref{lem:menger_vertex}.

The statement in \eqref{item:menger subgroup} can be deduced  from \eqref{item:menger subcomplex} thanks to  \cite[Theorem 0.1.(2)]{DS21}. We only need to check that  $L$ is not a join.
By contradiction, say $L$ is the join of two subcomplexes $L_1,L_2$. 
Since $L$ has no induced squares and embeds in a surface, we have that (up to reversing the roles) $L_1$ is a single vertex and $L_2$ embeds in a circle. In particular, it follows that $L$ embeds in a disk, contradicting the fact that $L$ is non-planar.
\end{proof}

Since $\Gamma^4$ and $\Gamma^6$ are RACGs defined by flag-no-square triangulations of a torus, we can use Proposition~\ref{prop:menger_surface} with $n=1$ to obtain the following.
\begin{theorem}\label{thm:menger}
    Both $\Gamma^4$ and $\Gamma^6$ contain special subgroups with limit set a Menger curve.
\end{theorem}

Moreover, we can find smaller special subgroups with limit set a Menger curve, as follows.

The complex defining  $\Gamma^6$ 
has a collection of $7$ pairwise non-adjacent vertices,  given by the  $7$  centers of the hexagonal tiling of the defining complex; see Figure~\ref{fig:menger}. 
By Proposition~\ref{prop:menger_surface}, the special subgroup generated by the other $14$ vertices has Gromov boundary homeomorphic to the Menger curve.
Since the full subcomplex on these $14$ vertices is $1$-dimensional (i.e., it is a triangle-free graph) the Gromov boundary could also be computed with  \cite[Corollary 1.7]{DHW23}. 
Note that this graph has  $K_{3,3}$ as a minor, i.e., one can obtain $K_{3,3}$ from this graph by deleting edges, deleting vertices, and  contracting edges. So, this graph is non--planar by Wagner's theorem.

\begin{figure}[ht]
    \centering
    \includegraphics[scale=1.1]{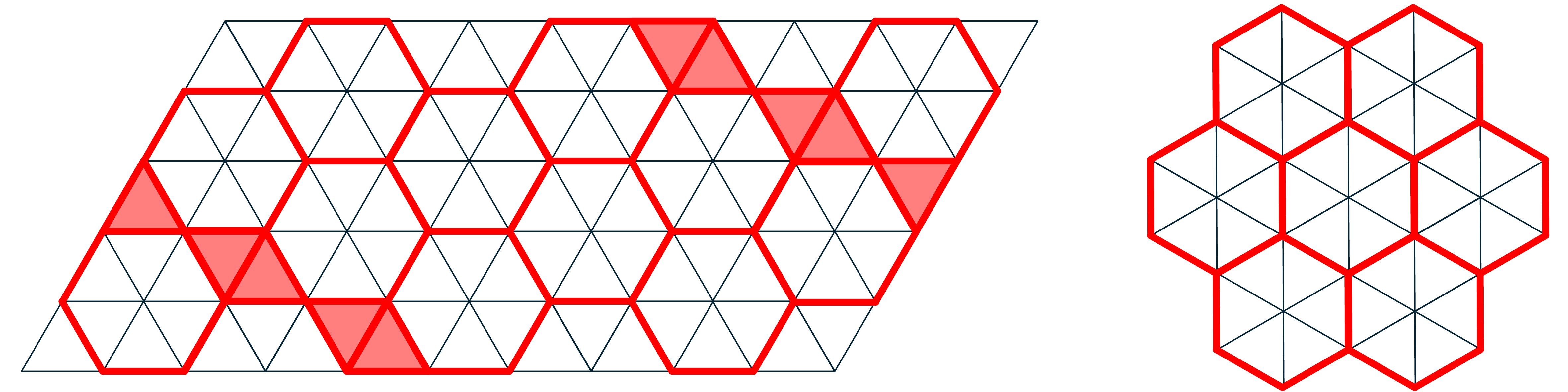}
    \caption{A special subgroup of $\Gamma^d$ with limit set a Menger curve ($d=4$ left, $d=6$ right).}
    \label{fig:menger}
\end{figure}

The complex defining  $\Gamma^4$ has a collection of $15$ pairwise non-adjacent vertices, which one can find by trying to fit a hexagonal grid in this complex; see Figure~\ref{fig:menger}. 
By Proposition~\ref{prop:menger_surface}, the special subgroup generated by the other $35$ vertices has Gromov boundary homeomorphic to the Menger curve.

We note  that, in contrast to the case of $\Gamma^6$, one cannot directly apply \cite[Corollary 1.7]{DHW23} to find a special subgroup with limit set a Menger curve, because the complex $K$ defining  $\Gamma^4$ does not admit a triangle-free inseparable full subgraph.
This can be seen as follows.
Since every vertex in $K$ has degree $6$, in any triangle-free subgraph $L\subseteq K$, every vertex can have degree at most $3$.
On the other hand, if $L$ has a vertex $v$ of degree $1$ or $2$, then the link of $v$ in $L$ is a global cut vertex or a global cut pair respectively.
It follows that every vertex of $L$ must have degree exactly $3$.
In particular, every edge of $L$ is contained in $2$ edge-loops of length $6$, so the complementary regions of $L$ in $K$ are given by regular hexagons. However, $K$ is not tiled by hexagons, because the number of triangles is $100$, which is not a multiple of $6$.

\begin{remark}\label{remark:bourdon}
Bourdon \cite{BO97} previously constructed examples of convex cocompact subgroups of $\mathrm{Isom} (\hh^4)$ with limit set a Menger curve. It was explained to the first-named author by Fran\c{c}ois Gu\'eritaud that at least some of Bourdon's examples are (perhaps up to deformation) commensurable to hyperbolic reflection groups, as is illustrated by the following example.

Consider the Coxeter group $W$ with the following Coxeter diagram.\footnote{Here, we use the classical conventions for Coxeter diagrams, so that an edge between two vertices indicates that the product of the corresponding standard generators has order $3$, an edge with label $\infty$ between two vertices indicates that the product of the corresponding standard generators has infinite order, and the lack of an edge between two vertices indicates that the corresponding standard generators commute.}

\begin{figure}[ht!]\begin{tikzpicture}[thick,scale=0.8, every node/.style={transform shape}]
\newcommand{\unit}{0.75}
\node[draw,circle, inner sep=2pt, minimum size=2pt] (1) at (-5*\unit,0) {{1}};
\node[draw,circle, inner sep=2pt, minimum size=2pt] (3) at (-3*\unit,0) {{3}};
\node[draw,circle, inner sep=2pt, minimum size=2pt] (5) at (-\unit,0) {{5}};
\node[draw,circle, inner sep=2pt, minimum size=2pt] (6) at (\unit,0) {{6}};
\node[draw,circle, inner sep=2pt, minimum size=2pt] (4) at (3*\unit,0) {{4}};
\node[draw,circle, inner sep=2pt, minimum size=2pt] (2) at (5*\unit,0) {{2}};

\draw (1) -- (3) ;
\draw (3) -- (5) node[above,midway] {$\infty$};
\draw (5) -- (6);
\draw (6) -- (4) node[above,midway] {$\infty$};
\draw (4) -- (2);
\end{tikzpicture}
\end{figure}

Denoting by $s_1, \ldots, s_6$ the corresponding standard generators, consider the index-$6$ reflection subgroup $W'$ of $W$ generated by $s_1$, $s_2$, $t_1 := s_3$, $t_2:= s_4$, $t_3:= s_5s_3s_5$, $t_4:= (s_5s_6s_5)s_4(s_5s_6s_5)$, $t_5:= (s_6s_5s_6)s_3(s_6s_5s_6)$, and $t_6:= s_6s_4s_6$ (note that the dihedral subgroup of $W$ generated by $s_5$ and $s_6$ is a retract of $W$, and $W'$ is nothing but the kernel of the retraction map from $W$ onto this dihedral subgroup). The subgroup $\Gamma < W'$ generated by $s_1t_{2k+1}$ and $s_2t_{2k+2}$ for $0 \leq k \leq 2$ is of index $4$ in $W'$ and is isomorphic to the group $\Gamma_{63}$ in Bourdon's notation \cite{BO97}; the latter is, as demonstrated by Benakli \cite{BE92}, a Gromov hyperbolic group with Gromov boundary a Menger curve. Now the group $W$ can be realized as a convex cocompact reflection group in $\mathrm{Isom} (\hh^4)$; indeed, by considering the signature of the associated Gram matrix, one sees that there is a unique way to do this (up to conjugation) so that the hyperbolic distance between the fixed hyperplane of $s_3$ and the fixed hyperplane of $s_5$ is equal to the distance between the fixed hyperplane of $s_4$ and the fixed hyperplane of $s_6$ (and in this case, this distance is $\mathrm{arccosh}\left(\tfrac{3}{4}\sqrt{2}\right)$. The resulting representation of the subgroup $\Gamma < W$ is of the form described in~\cite{BO97}. 

\end{remark}

\printbibliography

\end{document}